\documentclass[12pt]{article}
\usepackage{times}
\usepackage[a4paper,text={128mm,185mm}, centering]{geometry}
\usepackage{changepage}

\usepackage{amsthm}
\usepackage{fancyhdr}
\theoremstyle{plain}
\newtheorem{theorem}{Theorem}[section]
\newtheorem{lemma}[theorem]{Lemma}
\newtheorem{proposition}[theorem]{Proposition}

\newtheorem{definition}[theorem]{Definition}

\theoremstyle{definition}

\newtheorem{remark}[theorem]{Remark}

\usepackage[matrix,arrow,curve,cmtip]{xy}
\usepackage{amssymb, amsmath}
\usepackage{latexsym}
\usepackage{graphicx}

\newcommand{\Eq}{ \ensuremath{\mathrm{Eq}} }

\newdir{ >}{
  @{}*!/-10pt/@{>} }

\newdir{ |>}{
@{}*!/-5.5pt/@{|}*!/-10pt/:(1,-.2)@^{>}*!/-10pt/:(1,+.2)@_{>} }

\title{\vskip 5pt  \bf  GOURSAT COMPLETIONS}
\author{\itshape\bfseries { Diana RODELO and Idriss TCHOFFO NGUEFEU}}
\date{}

\begin{document}
\maketitle

\thispagestyle{empty}
\vskip 25pt
\begin{adjustwidth}{0.5cm}{0.5cm}
{\small
{\bf Abstract.} We characterize categories with weak finite limits whose regular completions give rise to Goursat categories.
}
\end{adjustwidth}


\section{Introduction}

The construction of the free exact category over a category with finite limits was introduced in ~\cite{FECLEO}. It was later improved to the construction of the free exact category over a category with finite weak limits (\emph{weakly lex}) in \cite{REC}. This followed from the fact that the uniqueness of the finite limits of the original category is never used in the construction; only the existence. In ~\cite{REC}, the authors also considered the free regular category over a weakly lex one.

An important property of the free exact (or regular) construction is that such categories always have enough (regular) projectives. In fact, an exact category $\mathbb{A}$ may be seen as the exact completion of a weakly lex category if and only if it has enough projectives. If so, then $\mathbb{A}$ is the exact completion of any of its \emph{projective covers}. Such a phenomenon is captured by varieties of universal algebras: they are the exact completions of their full subcategory of free algebras.

Having this link in mind, our main interest in studying this subject is to characterize projective covers of certain algebraic categories through simpler properties involving projectives and to relate those properties to the known varietal characterizations in terms of the existence of operations of their varietal theories (when it is the case). Such kind of studies have been done for the projective covers of categories which are: Mal'tsev~\cite{ECRAC}, protomodular and semi-abelian~\cite{G}, (strongly) unital and subtractive~\cite{compl}.

The aim of this work is to obtain characterizations of the weakly lex categories whose regular completion is a Goursat (=$3$-permutable) category (Propositions~\ref{cover} and ~\ref{weaksquare}). We then relate them to the existence of the quaternary operations which characterize the varieties of universal algebras which are $3$-permutable (Remark~\ref{vars}).


\section{Preliminaries}\label{pre}

In this section, we briefly recall some elementary categorical notions needed in the following.

A category with finite limits is \textbf{regular} if regular epimorphisms are stable under pullbacks, and kernel pairs have coequalizers. Equivalently, any arrow $f: A\longrightarrow B$ has a unique factorisation $f=i r $ (up to isomorphism), where $r$ is a regular epimorphism and $i$ is a monomorphism and this factorisation is pullback stable.

A \textbf{relation} $R$ from  $X$ to  $Y$  is a subobject $\langle r_1,r_2 \rangle : R \rightarrowtail X \times Y $. The opposite relation of $R$, denoted $R^o$, is the relation from $Y$ to $X$ given by the subobject $\langle r_2,r_1 \rangle : R \rightarrowtail Y \times X $. A relation $R$ from $X$ to $X$ is called a relation on $X$. We shall identify a morphism $f: X \longrightarrow Y$ with the relation $\langle 1_X,f \rangle: X \rightarrowtail X \times Y$ and write $f^o$ for its opposite relation. Given two relations $R \rightarrowtail X \times Y $ and $S \rightarrowtail Y \times Z $ in a regular category, we write $SR \rightarrowtail X \times Z $ for their relational composite. With the above notations, any relation $\langle r_1, r_2 \rangle: R \rightarrowtail X \times Y$ can be seen as the relational composite $r_2r_1^o$.
The following properties are well known and easy to prove (see \cite{ckp} for instance); we collect them in the following lemma:
\begin{lemma}
\label{lem}
Let $f: X \longrightarrow Y$ be an arrow in a regular category  $\mathbb{C}$, and let  $ f=i r $ be its (regular epimorphism, monomorphism) factorisation. Then:
\begin{enumerate}
  \item $f^of$ is the kernel pair of $f$, thus $1_X \leqslant f^of $; moreover, $1_X = f^of$ if and only if $f$ is a monomorphism;
  \item $ff^o$ is $(i,i)$, thus $ff^o \leqslant 1_Y$; moreover, $ff^o = 1_Y$ if and only if $f$ is a  regular epimorphism;
  \item $ff^of = f $ and $f^off^o = f^o$.
\end{enumerate}
\end{lemma}

A relation $R$ on $X$ is \textbf{reflexive} if $1_X \leqslant R$, \textbf{symmetric} if $R^o \leqslant R$, and \textbf{transitive} if $RR \leqslant R$. As usual, a relation $R$ on $X$ is an \textbf{equivalence relation} when it is reflexive, symmetric and transitive. In particular, a kernel pair $\langle f_1,f_2 \rangle: \Eq(f)\rightarrowtail X \times X$ of a morphism $f: X \longrightarrow Y$ is an equivalence relation.

By dropping the assumption of uniqueness of the factorization in the definition of a limit, one obtains the definition of a weak limit. We call \textbf{weakly lex} a category with weak finite limits.

An object $P$ in a category is (regular) \textbf{projective} if, for any arrow $f: P \longrightarrow X$ and for any regular epimorphism $g: Y \twoheadrightarrow X$ there exists an arrow $h: P \longrightarrow Y$ such that $g h = f $. We say that a full subcategory $\mathbb{C}$ of $\mathbb{A}$ is a \textbf{projective cover} of $\mathbb{A}$ if two conditions are satisfied:
\begin{itemize}
\item any object of $\mathbb{C}$ is regular projective in $\mathbb{A}$;
\item for any object $X$ in $\mathbb{A}$, there exists a ($\mathbb{C}$-)cover of $X$, that is an object $C$ in $\mathbb{C}$ and a regular epimorphism $C\twoheadrightarrow X$.
\end{itemize}

When $\mathbb{A}$ admits a projective cover, one says that $\mathbb{A}$ has \textit{enough projectives}.
\begin{remark}\label{weaklims}
If $\mathbb{C}$ is a projective cover of a weakly lex category $\mathbb{A}$, then $\mathbb{C}$ is also weakly lex~\cite{REC}. For example, let $X$ and $Y$ be objects in $\mathbb{C}$ and $\xymatrix@C=15pt{ X & W \ar[l] \ar[r] & Y}$ a weak product of $X$ and $Y$ in $\mathbb{A}$. Then, for any cover $\bar{W}\twoheadrightarrow W$ of $W$, $\xymatrix@C=15pt{ X & \bar{W} \ar[l] \ar[r] & Y}$ is a weak product of $X$ and $Y$ in $\mathbb{C}$. Furthermore, if $\mathbb{A}$ is a regular category, then the induced morphism $W\twoheadrightarrow X\times Y$ is a regular epimorphism.
Similar remarks apply to all weak finite limits.
\end{remark}


\section{Goursat categories}
In this section we review the notion of Goursat category and the characterizations of Goursat categories through regular images of equivalence relations and through Goursat pushouts.

\begin{definition} \emph{\cite{CLP,ckp}}
A regular category $\mathbb{C}$ is called a \textbf{Goursat category} when the equivalence relations in $\mathbb{C}$ are $3$-permutable, i.e. $RSR = SRS$ for any pair of equivalence relations $R$ and $S$ on the same object.
\end{definition}

When $\mathbb{C}$ is a regular category, $(R,r_1,r_2)$ is an equivalence relation on $X$ and $f: X \twoheadrightarrow Y$ is a regular epimorphism, we define the \textbf{regular image of $R$ along $f$} to be the relation $f(R)$ on $Y$ induced by the (regular epimorphism, monomorphism) factorization $\langle s_1, s_2 \rangle \psi$ of the composite $(f\times f) \langle r_1,r_2\rangle$:
\[
      \xymatrix{
R \ar@{.>>}[r]^{\psi} \ar@{ >->}[d]_{\langle r_1,r_2\rangle} & f(R)  \ar@{ >.>}[d]^{\langle s_1,s_2\rangle}\\
X \times X \ar@{>>}[r]_{f \times f} & Y \times Y.
  }
\]
Note that the regular image $f(R)$ can be obtained as the relational composite $f(R)=fRf^o=fr_2r_1^of^o$. When $R$ is an equivalence relation, $f(R)$ is also reflexive and symmetric. In a general regular category $f(R)$ is not necessarily an equivalence relation.
This is the case in a \emph{Goursat category} according to the following theorem.

\begin{theorem}\label{CKP} \emph{\cite{ckp}}  A regular category $\mathbb{C}$ is a Goursat category if and only if for any regular epimorphism $f: X \twoheadrightarrow Y$ and any equivalence relation $R$ on $X$, the regular image $f(R)= fRf^o$ of $R$ along $f$ is an equivalence relation.
\end{theorem}

Goursat categories are well known in Universal Algebra. In fact, by a classical theorem in \cite{hm}, a variety of universal algebras is a Goursat category precisely when its theory has two quaternary operations $p$ and $q$ such that the identities $p(x,y,y,z)= x$, $q(x,y,y,z)= z$ and $p(x,x,y,y)= q(x,x,y,y)$ hold. Accordingly, the varieties of groups, Heyting algebras and implication algebras are Goursat categories. The category of topological group, Hausdorff groups, right complemented semi-group  are also Goursat categories.

There are many known characterizations of Goursat categories (see \cite{ckp,gr,grod,grt} for instance). In particular the following characterization, through Goursat pushouts,  will be useful:
\begin{theorem}\label{Goursatpushout}\emph{\cite{gr}}
Let $\mathbb{C}$ be a regular category. The following conditions are equivalent:
\begin{enumerate}
 \item[(i)] $\mathbb{C}$ is a Goursat category;
 \item[(ii)] any commutative diagram of type \emph{($\mathrm{I}$)} in $\mathbb{C}$, where $\alpha$ and $\beta$ are regular epimorphisms and $f$ and $g$ are split epimorphisms
  \[
    \xymatrix@C=2cm{
   X \ar@{}[dr]|{(\mathrm{I})} \ar@{->>}[r]^{\alpha} \ar@<3pt>[d]^{f}  & U \ar@<3pt>[d]^{g} \ar@<40pt>@{}[d]^(.3){g\alpha=\beta f} \ar@<40pt>@{}[d]^(.7){\alpha s=t\beta} \\ Y \ar@{->>}[r]_{\beta} \ar@<3pt>[u]^{s}& W, \ar@<3pt>[u]^t
  }
\]
(which is necessarily a pushout) is a \textbf{Goursat pushout}: the morphism $ \lambda : \Eq(f) \longrightarrow \Eq(g)$, induced by the universal property of kernel pair $\Eq(g)$ of $g$, is a regular epimorphism.
\end{enumerate}
\end{theorem}

\begin{remark}
\label{variant} Diagram ($\mathrm{I}$) is a Goursat pushout precisely when the regular image of $\Eq(f)$ along $\alpha$ is (isomorphic to) $\Eq(g)$. From Theorem~\ref{Goursatpushout}, it then follows that a regular category $\mathbb{C}$ is a Goursat category if and only if for any commutative diagram of type ($\mathrm{I}$) one has $\alpha(\Eq(f))= \Eq(g)$.
\end{remark}

Note that Theorem~\ref{CKP} characterizes Goursat categories through the property that regular images of equivalence relations are equivalence relations, while Theorem~\ref{Goursatpushout} characterizes them through the property that regular images of certain kernel pairs are kernel pairs.


\section{Projective covers of Goursat categories}
In this section, we characterize the categories with weak finite limits whose regular completion are Goursat categories.
\begin{definition} Let $\mathbb{C}$ be a weakly lex category:
\begin{enumerate}
\item a \textbf{pseudo-relation} on an object X of $\mathbb{C}$ is a pair of parallel arrows $\xymatrix {R \ar@<2pt>[r]^{r_1} \ar@<-2pt>[r]_{r_2} & X;}$ a pseudo-relation is a relation if $r_1$ and $r_2$ are jointly monomorphic;
\item a pseudo-relation $ \xymatrix {R \ar@<2pt>[r]^{r_1} \ar@<-2pt>[r]_{r_2} & X}$ on  $X$  is said to be:
\begin{itemize}
  \item  \textbf{reflexive} when there is an arrow $ r: X \longrightarrow R $ such that $ r_1 r = 1_X = r_2 r $;
  \item  \textbf{symmetric} when there is an arrow  $ \sigma :  R \longrightarrow R $ such that $ r_2 = r_1 \sigma $ and $r_1 = r_2 \sigma $;
  \item  \textbf{transitive} if by considering a weak pullback
  $$
     \xymatrix{
     W \ar[r]^-{p_2}  \ar[d]_{p_1}  & R \ar[d]^{r_1} \\ R \ar[r]_{r_2}& X,
  }
$$
there is an arrow $ t : W \longrightarrow R$ such that $ r_1 t = r_1 p_1 $ and $r_2  t = r_2 p_2$.
 \item a \textbf{pseudo-equivalence relation} if it is  reflexive, symmetric and transitive.
\end{itemize}
\end{enumerate}
\end{definition}

Remark that the transitivity of a pseudo-relation $\xymatrix {R \ar@<2pt>[r]^{r_1} \ar@<-2pt>[r]_{r_2} & X}$ does not depend on the choice of the weak pullback of $r_1$ and $r_2$; in fact, if
$$
     \xymatrix{
     \bar{W} \ar[r]^-{\bar{p_2}}  \ar[d]_{\bar{p_1}}  & R \ar[d]^{r_1} \\ R \ar[r]_{r_2}& X,
  }
$$
is another weak pullback, the factorization $\bar{W} \longrightarrow W$ composed with the transitivity $t: W \longrightarrow R$ ensures that the pseudo-relation is transitive also with respect to the second weak pullback.
\vspace{0.5cm}

The following property from \cite{Enrico} (Proposition 1.1.9) will be useful in the sequel:

\begin{proposition}\emph{\cite{Enrico}}\label{EV}
Let $\mathbb{C}$ be a projective cover of a regular category $\mathbb{A}$. Let $ \xymatrix {R \ar@<2pt>[r]^{r_1} \ar@<-2pt>[r]_{r_2} & X}$ be a pseudo-relation in $\mathbb{C}$ and consider its (regular epimorphism, monomorphism) factorization in $\mathbb{A}$
$$
     \xymatrix{
 R  \ar[rr]^{(r_1, r_2)} \ar@{->>}[rd]_p & {} & X \times X. \\ {} & E \ar@{ >->}[ru]_{(e_1, e_2)} & {}
  }
$$
Then, $R$ is a pseudo-equivalence relation in $\mathbb{C}$ if and only if $S$ is an equivalence relation in $\mathbb{A}$.
\end{proposition}

\begin{definition}
Let $\mathbb{C}$ be a weakly lex category. We call $\mathbb{C}$ a \textbf{weak Goursat category} if, for any pseudo-equivalence relation $ \xymatrix { R \ar@<2pt>[r]^{r_1} \ar@<-2pt>[r]_{r_2} & X}$  and any regular epimorphism $f: X \twoheadrightarrow Y$, the composite $\xymatrix { R \ar@<2pt>[r]^{fr_1} \ar@<-2pt>[r]_{fr_2} & X}$ is also a pseudo-equivalence relation.
\end{definition}

We use Remark~\ref{weaklims} repeatedly in the next results.

\begin{proposition}\label{cover}
Let $\mathbb{C}$ be a projective cover of a regular category $\mathbb{A}$. Then $\mathbb{A}$ is a Goursat category if and only if $\mathbb{C}$ is a weak Goursat category.
\end{proposition}

\begin{proof}
Since $\mathbb{C}$ is a projective cover of a regular category $\mathbb{A}$, then $\mathbb{C}$ is weakly lex.

Suppose that $\mathbb{A}$ is a Goursat category. Let $ \xymatrix { R \ar@<2pt>[r]^{r_1} \ar@<-2pt>[r]_{r_2} & X}$ be a pseudo-equiva\-lence relation in $\mathbb{C}$ and let $f: X \twoheadrightarrow Y$ be a regular epimorphism in $\mathbb{C}$. For the (regular epimorphism, monomorphism) factorizations of $\langle r_1, r_2\rangle$ and $\langle fr_1, fr_2\rangle$ we get the following diagram
\begin{equation}
\label{proj}
    \vcenter{\xymatrix {
   R \ar[rr]^(0.4){\langle r_1, r_2\rangle}  \ar@{=}[ddd] \ar@{->>}[rd]_{p} & {}   & X\times X \ar@{->>}[ddd]^{f \times f} \\
   {} & E  \ar@{ >->}[ru]_(0.4){\langle e_1, e_2\rangle} \ar@{.>}[d]_{w} & {} \\
   {} & S  \ar@{ >->}[rd]^(0.4){\langle s_1, s_2\rangle}  & {} \\
   R \ar@{->>}[ru]^{q} \ar[rr]_(0.4){\langle fr_1, fr_2\rangle} & {} & Y \times Y,
}}
\end{equation}

where $w:E\longrightarrow S$ is induced by the strong epimorphism $p$
\[
      \xymatrix@C=1cm {
   R \ar@{->>}[r]^{p} \ar@{->>}[d]_{q} & E \ar[d]^{(f \times f) \langle e_1, e_2\rangle} \ar@{.>}[dl]_{w} \\ S  \ar@{ >->}[r]_{\langle s_1,s_2\rangle} & Y \times Y.
}
\]
Then $w$ is a regular epimorphism and by the commutativity of the right side of $\eqref{proj}$, one has $S = f(E)$.
By Proposition \ref{EV}, we know that $E$ is an equivalence relation in $\mathbb{A}$. Since $\mathbb{A}$ is a Goursat category, then $S = f(E)$ is also an equivalence relation in $\mathbb{A}$ and by Proposition \ref{EV}, we can conclude that $ \xymatrix { R \ar@<2pt>[r]^{fr_1} \ar@<-2pt>[r]_{fr_2} & X}$ is a pseudo-equivalence relation in $\mathbb{C}$.

Conversely, suppose that $\mathbb{C}$ is a weak Goursat category. Let $ \xymatrix { R \ar@<2pt>[r]^{r_1} \ar@<-2pt>[r]_{r_2} & X}$ be an equivalence relation in $\mathbb{A}$ and $f :X \twoheadrightarrow Y$ a regular epimorphism. We are going to show that $f(R) = S$
$$ \xymatrix {
   R  \ar@{->>}[r]^-{h} \ar@<.5ex>[d]^{r_2} \ar@<-.5ex>[d]_{r_1}  & f(R)=S \ar@<.5ex>[d]^{s_2} \ar@<-.5ex>[d]_{s_1} \\ X  \ar@{->>}[r]_{f} & Y
}
$$
is an equivalence relation; it is obviously reflexive and symmetric. In order to conclude that $\mathbb{A}$ is a Goursat category,  we must prove that $S$ is transitive, i.e that $S$ is an equivalence relation.

We begin by covering the regular epimorphism $f$ in $\mathbb{A}$ with a regular epimorphism $\bar{f}$ in $\mathbb{C}$. For that we take the cover $y : \bar{Y} \twoheadrightarrow Y$, consider the pullback of $y$ and $f$ in $\mathbb{A}$ and take its cover $\alpha: \bar{X} \twoheadrightarrow X \times_Y \bar{Y}$
\[
     \xymatrix@C=1cm{
  \bar{X} \ar@{->>}[rd]^{\alpha} \ar@/^2pc/@{->>}[rrd]^{\bar{f}} \ar@/_2pc/@{->>}[ddr]_{x} & {} & {}\\ {} & X \times_Y \bar{Y} \ar@{->>}[r]^{f'}  \ar@{->>}[d]_{y'}& \bar{Y} \ar@{->>}[d]^{y}\\ {} & X  \ar@{->>}[r]_{f} & Y.
  }
\]
Note that the above outer diagram is a \emph{regular pushout}, so that
\begin{equation}\label{push}f^o y = x \bar{f}^o\;\;\mathrm{and}\;\; y^of=\bar{f}x^o \end{equation}
(Proposition 2.1 in~\cite{ckp}).

Next, we take the inverse image $x^{-1}(R)$ in $\mathbb{A}$, which in an equivalence relation since $R$ is, and cover it to obtain a pseudo-equivalence $W \rightrightarrows \bar{X}$ in $\mathbb{C}$. By assumption $ \xymatrix { W \ar@<2pt>[r] \ar@<-2pt>[r] & \bar{X} \ar@{->>}[r]^{\bar{f}} & \bar{Y}}$ is a pseudo-equivalence relation in $\mathbb{C}$ so it factors through an equivalence relation, say $ \xymatrix { V \ar@<2pt>[r]^{v_1} \ar@<-2pt>[r]_{v_2} & \bar{Y},}$ in $\mathbb{A}$. We have

$$
\xymatrix@C=25pt@R=15pt{
W \ar@{=}[rrr] \ar@{->>}[d]_{w} &{} &{} & W \ar[ddd] \ar@{->>}[rd]^v &{} & {} \\
x^{-1}(R)\ar@{ >->}[dd]_{\langle\rho_1, \rho_2\rangle} \ar@{->>}[dr]_{\pi_R} \ar@{.>}[rrrr]^-{\gamma} & {} & {}& {} &  V \ar@{ >->}[ddl]_(0.3){\langle v_1, v_2\rangle} \ar@{.>>}[dr]^{\lambda} & {} \\
{}  & R \ar@{ >->}[dd]_(.3){\langle r_1,r_2\rangle} \ar@{->>}[rrrr]^(.2){h} & {} & {} & {} & S \ar@{ >->}[dd]^-{\langle s_1,s_2\rangle}  \\
\bar{X} \times \bar{X} \ar@{->>}[dr]_-{x \times x} \ar@{-->>}[rrr]^(.7){\bar{f} \times \bar{f}} & {} & {} & \bar{Y} \times \bar{Y} \ar@{->>}[rrd]^{y \times y} & {} & {} \\
{}  & X \times X  \ar@{->>}[rrrr]_{f \times f} & {} & {} &{} & Y \times Y, }
$$
where $\gamma$ and $\lambda$ are induced by the strong epimorphisms $w$ and $v$, respectively\\
\begin{center}
$
      \xymatrix@C=2cm {
   W \ar@{->>}[r]^w  \ar@{->>}[dd]_v & x^{-1}(R) \ar@{ >->}[d]^{\langle\rho_1, \rho_2\rangle} \ar@{.>}[ddl]_{\gamma} \\ {} & \bar{X} \times \bar{X} \ar@{->>}[d]^{\bar{f} \times \bar{f}} \\ V  \ar@{ >->}[r]_{\langle v_1,v_2\rangle} & \bar{Y} \times \bar{Y}
}
$
and
$
      \xymatrix@C=2cm {
   W \ar@{->>}[r]^v  \ar@{->>}[dd]_{h\pi_R w} & V \ar@{ >->}[d]^{\langle v_1, v_2\rangle} \ar@{.>}[ddl]_{\lambda} \\ {} & \bar{Y} \times \bar{Y} \ar@{->>}[d]^{y \times y} \\ S  \ar@{ >->}[r]_{\langle s_1,s_2\rangle} & Y \times Y.
}
$
\end{center}

Since $\gamma$ is a regular epimorphism, we have $V = \bar{f}(x^{-1}(R))$.
Since $\lambda$ is a regular epimorphism, we have $S = y(V)$. One also has $V = y^{-1}(S)$ because
$$
\begin{matrix}
 y^{-1}(S) &=& y^o S y & \\
 {} &=& y^o f(R) y &\\
{} &=& y^o f R f^o y &\\
{} &=& \bar{f} x^o R x \bar{f}^o & \text{(by \eqref{push})}\\
{} &=& \bar{f}(x^{-1}(R))&\\
{} &=& V.&
\end{matrix}
$$

Finally, $S$ is transitive since
$$
\begin{matrix}
 SS &=& yy^o S yy^o S yy^o &\text{(Lemma~\ref{lem}(2))} \\
 {} &=& yy^{-1}(S) y^{-1}(S) y^o &\\
{} &=& y VV y^o &\\
{} &=& y V y^o & \text{(since V is an equivalence relation)}\\
{} &=& y(V)&\\
{} &=& S.&
\end{matrix}
$$

\end{proof}

We may also consider weak Goursat categories through a property which is more similar to the one mentioned in Theorem~\ref{CKP}:

\begin{lemma} Let $\mathbb{C}$ be a projective cover of a regular category $\mathbb{A}$. Then $\mathbb{C}$ is a weak Goursat category if and only if for any commutative diagram in $\mathbb{C}$
\begin{equation}
\label{equivdef}
    \vcenter{ \xymatrix {
   R  \ar@{->>}[r]^{\varphi} \ar@<.5ex>[d]^{r_2} \ar@<-.5ex>[d]_{r_1}  & S \ar@<.5ex>[d]^{s_2} \ar@<-.5ex>[d]_{s_1} \\ X  \ar@{->>}[r]_{f} & Y
}}
\end{equation}
such that $f$ and $\varphi$ are regular epimorphism and $R$ is a pseudo-equivalence relation, then $S$ is a pseudo-equivalence relation.
\end{lemma}

\begin{proof} $(i) \Rightarrow (ii)$ Since $\xymatrix { R \ar@<2pt>[r]^{r_1} \ar@<-2pt>[r]_{r_2} & X}$ is a pseudo-equivalence relation, by assumption $\xymatrix { R \ar@<2pt>[r]^{fr_1} \ar@<-2pt>[r]_{fr_2} & X}$ is also a pseudo-equivalence relation and then its (regular epimorphism, monomorphism)  factorization gives an equivalence relation $\xymatrix { E \ar@<2pt>[r]^{e_1} \ar@<-2pt>[r]_{e_2} & Y}$ in $\mathbb{A}$ (Proposition~\ref{EV}). We have the following commutative diagram
$$ \xymatrix {
   R  \ar@{=}[r] \ar@<.5ex>[dd]^{r_2} \ar@<-.5ex>[dd]_{r_1}  & R \ar@<.5ex>[dd]^(.4){f r_2} \ar@<-.5ex>[dd]_(.4){f r_1} \ar@{->>}[rr]^{\varphi} \ar@{->>}[rd]^{\rho} & {} & S \ar@{.>}[dl]^{\sigma}  \ar@<.5ex>@/^4pc/[ddll]^{s_2} \ar@<-.5ex>@/^4pc/[ddll]_{s_1} \\
   {} & {} & E \ar@<.5ex>[dl]^{e_2} \ar@<-.5ex>[dl]_{e_1} & {} \\
   X  \ar@{->>}[r]_{f} & Y & {} & {}
}
$$
where $\sigma: S \longrightarrow E$ is induced by the strong epimorphism $\varphi$
\[
      \xymatrix@C=1cm {
   R \ar@{->>}[r]^{\varphi} \ar@{->>}[d]_{\rho} & S \ar[d]^{\langle s_1, s_2\rangle} \ar@{.>}[dl]_{\sigma} \\ E  \ar@{ >->}[r]_{\langle e_1,e_2\rangle} & Y \times Y.
}
\]
Then $\sigma$ is a regular epimorphism and $\xymatrix { S \ar@<2pt>[r]^{s_1} \ar@<-2pt>[r]_{s_2} & Y}$ is a pseudo-equivalence relation (Proposition~\ref{EV}).

 $(ii) \Rightarrow (i)$ Let $ \xymatrix { R \ar@<2pt>[r]^{r_1} \ar@<-2pt>[r]_{r_2} & X}$ be a pseudo-equivalence relation in $\mathbb{C}$ and $f: X \twoheadrightarrow Y$ a regular epimorphism. The following diagram is of the type \eqref{equivdef}
$$
\xymatrix {
   R  \ar@{=}[r] \ar@<.5ex>[d]^{r_2} \ar@<-.5ex>[d]_{r_1}  & R \ar@<.5ex>[d]^{fr_2} \ar@<-.5ex>[d]_{fr_1} \\ X  \ar@{->>}[r]_{f} & Y.
}
$$
Since $\xymatrix { R \ar@<2pt>[r]^{r_1} \ar@<-2pt>[r]_{r_2} & X}$ is a pseudo-equivalence relation, then  by assumption\\ $\xymatrix { R \ar@<2pt>[r]^{f r_1} \ar@<-2pt>[r]_{f r_2} & Y}$ is also a pseudo-equivalence relation.
\end{proof}

Alternatively, weak Goursat categories may be characterized through a property more similar to the one mentioned in Remark~\ref{variant}:

\begin{proposition}\label{weaksquare} Let $\mathbb{C}$ be a projective cover of a regular category $\mathbb{A}$. The following conditions are equivalent:
\begin{enumerate}
 \item[(i)] $\mathbb{A}$ is a Goursat category;
 \item[(ii)] $\mathbb{C}$ is a weak Goursat category;
 \item[(iii)] For any commutative diagram of type \emph{($\mathrm{I}$)} in $\mathbb{C}$ where
$$\xymatrix@C=2cm{
   F \ar@<.5ex>[d]^{\beta_2} \ar@<-.5ex>[d]_{\beta_1} \ar@{->>}[r]^{\lambda}& G \ar@<.5ex>[d]^{\rho_2} \ar@<-.5ex>[d]_{\rho_1} \\
   X \ar@{}[dr]|{\mathrm{(I)}} \ar@{->>}[r]^{\alpha} \ar@<3pt>[d]^{f}  & U \ar@<3pt>[d]^{g} \\
   Y \ar@{->>}[r]_{\beta} \ar@<3pt>[u]^{s} & W \ar@<3pt>[u]^t
  }
$$
$F$ is a weak kernel pair of $f$ and $\lambda$ is a regular epimorphism (in $\mathbb{C}$), then $G$ is a weak kernel pair of $g$.
\end{enumerate}
\end{proposition}

\begin{proof}
$(i) \Leftrightarrow (ii)$ By Proposition \ref{cover}.

$(i) \Rightarrow (iii)$  If we take the kernel pairs of $f$ and $g$, then the induced morphism $\bar{\alpha}:\Eq(f)\longrightarrow \Eq(g)$ is a regular epimorphism by Theorem~\ref{Goursatpushout}. Moreover, the induced morphism $\varphi:F\longrightarrow \Eq(f)$ is also a regular epimorphism. We get
\[
    \vcenter{\xymatrix{
   F \ar@{->>}[r]^{\lambda} \ar@{->>}[d]^{\varphi} & G \ar@{.>}[d]^{\omega} \ar@<.5ex>@/^4pc/[dd]^{\rho_2} \ar@<-.5ex>@/^4pc/[dd]_{\rho_1} \\
   \Eq(f) \ar@<.5ex>[d]^{f_2} \ar@<-.5ex>[d]_{f_1} \ar@{->>}[r]^{\bar{\alpha}} & \Eq(g) \ar@<.5ex>[d]^{g_2} \ar@<-.5ex>[d]_{g_1} \\
    X \ar@{->>}[r]^{\alpha} \ar@<3pt>[d]^{f}   & U \ar@<3pt>[d]^{g}
    \\ Y \ar@{->>}[r]_{\beta} \ar@<3pt>[u]^{s} & W, \ar@<3pt>[u]^t
  }}
\]
where $w:G\longrightarrow \Eq(g)$ is induced by the strong epimorphism $\lambda$
\[
      \xymatrix@C=1cm {
   F \ar@{->>}[r]^{\lambda} \ar@{->>}[d]_{\bar{\alpha}.\varphi} & G \ar[d]^{\langle\rho_1, \rho_2\rangle} \ar@{.>}[dl]_-{w} \\ \Eq(g)  \ar@{ >->}[r]_{\langle g_1,g_2\rangle} & U.
}
\]
This implies that $\omega$ is a regular epimorphism ($\omega \lambda = \bar{\alpha} \varphi$) and then $ \xymatrix@C=1cm {G \ar@<2pt>[r]^{\rho_1} \ar@<-2pt>[r]_{\rho_2} & U}$ is a weak kernel pair of $g$.

$(iii) \Rightarrow (ii)$
Consider the diagram \eqref{equivdef} in $\mathbb{C}$ where $ \xymatrix { R \ar@<2pt>[r]^{r_1} \ar@<-2pt>[r]_{r_2} & X}$ is a pseudo-equivalence relation. We want to prove that $ \xymatrix { S \ar@<2pt>[r]^{s_1} \ar@<-2pt>[r]_{s_2} & Y}$  is also a pseudo-equivalence. Take the (regular epimorphism, monomorphism) factorization of $R$ and $S$ in $\mathbb{A}$ and the induced morphism $\mu$ making the following diagram commutative
$$
\xymatrix {
   R  \ar@{->>}[rr]^{\varphi} \ar@<.5ex>[dd]^(.4){r_2} \ar@<-.5ex>[dd]_(.4){r_1} \ar@{->>}[rd]^{\rho}  & {} &  S \ar@{->>}[rd]^{\sigma} \ar@<.5ex>[dd]^(.3){s_2} \ar@<-.5ex>[dd]_(.3){s_1} & {} \\
   {} & U \ar@{.>}[rr]_(0.3){\mu} \ar@<.5ex>[dl]^{u_2} \ar@<-.5ex>[dl]_{u_1} & {} & V  \ar@<.5ex>[dl]^{v_2} \ar@<-.5ex>[dl]_{v_1} \\
    X  \ar@{->>}[rr]_{f}& {} & Y. & {}
}
$$
Since $\mu$ is a regular epimorphism, then $V = f(U)$ and consequently, $V$ is reflexive and symmetric.

Since $S$ is a pseudo-relation associated to $V$, then $S$ is also a reflexive and symmetric pseudo-relation. We just need to prove that $V$ is transitive. To do so, we apply our assumption to the diagram

$$
\xymatrix@C=25pt@R=20pt{
F \ar@{->>}[rr]^{\lambda} \ar@{->>}[dd]_{\delta} \ar@{->>}[rd] & {} & G \ar@{->>}[dd]^{\alpha} \\
{} & \Eq(r_1) \times_{\varphi(\Eq(r_1))} G \ar@{->>}[ru] \ar@{->>}[dl] & {} \\
 \Eq(r_1) \ar@{->>}[rr]_{\chi} \ar@<-.5ex>[d] \ar@<.5ex>[d] & {} & \varphi(\Eq(r_1))  \ar@<.5ex>[d] \ar@<-.5ex>[d] \\
 R \ar@{}[drr]|{\mathrm{(I)}} \ar@{->>}[rr]^{\varphi} \ar@<3pt>[d]^{r_1}& {}  & S \ar@<3pt>[d]^{s_1} \\
 X  \ar@{->>}[rr]_{f}  \ar@<3pt>[u]^{e_R}  & {} & Y \ar@<3pt>[u]^{e_S}
  }
$$
where $G$ is a  cover of the regular image $\varphi(\Eq(r_1))$ and $F$ is a cover of the pullback $\Eq(r_1) \times_{\varphi(Eq(r_1))} G$. Since $\delta$ is a regular epimorphism, then $ \xymatrix { F \ar@<2pt>[r] \ar@<-2pt>[r] & R}$ is a weak kernel pair of $r_1$. By assumption $ \xymatrix { G \ar@<2pt>[r] \ar@<-2pt>[r] & S}$ is a weak kernel pair of $s_1$, thus $\varphi(\Eq(r_1)) = \Eq(s_1)$. We then have
\[
\begin{matrix}
VV &=& v_2 v_1^o v_1 v_2^o  & \text{(since $V$ is symmetric}) \\
{}      &=& v_2 \sigma \sigma^o v_1^o v_1 \sigma \sigma^o v_2^o & \text{(Lemma~\ref{lem}(2))} \\
{}      &=& s_2 s_1^o s_1 s_2^o  & \text{($v_i\sigma=s_i$)}\\
{}      &=& s_2 \varphi r_1^o r_1 \varphi^o s_2^o  &\text{($\varphi(\Eq(r_1)) = \Eq(s_1)$)}\\
{}      &=& f r_2 r_1^o r_1 r_2^o f^o  & \text{($s_i \varphi = f r_i$ )}\\
{}      &=& f u_2 \rho \rho^o u_1^o u_1 \rho \rho^o u_2^o f^o  &  \text{($u_i\rho=r_i$)} \\
{}      &=& f u_2 u_1^o u_1 u_2^o f^o  & \text{(Lemma~\ref{lem}(2)) } \\
{}      &=& f UU f^o  & \text{(since $U$ is an equivalence relation in $\mathbb{A}$)} \\
{}      &=& f U f^o  &  \\
{}      &=& V. &
\end{matrix}
\]
\end{proof}
\begin{remark}\label{vars}
When $\mathbb{A}$ is a $3$-permutable variety and $\mathbb{C}$ its subcategory of free algebras, then the property stated in Proposition \ref{weaksquare} (iii) is precisely what is needed to obtain the existence of the quaternary operations $p$ and $q$ which characterize $3$-permutable varieties. Let $X$ denote the free algebra on one element. Diagram $\mathrm{(I)}$ below belongs to $\mathbb{C}$
$$
\xymatrix@C=2cm{
   F \ar@{=}[r] \ar@{->>}[d]^{\mu} & F \ar@{->>}[d]^{\lambda \mu}  \\
   \Eq(\nabla_2 + \nabla_2) \ar@<.5ex>[d]^{\pi_2} \ar@<-.5ex>[d]_{\pi_1} \ar[r]^-{\lambda} & \Eq(\nabla_3) \ar@<.5ex>[d] \ar@<-.5ex>[d] \\
    4X \ar@{}[dr]|{\mathrm{(I)}} \ar@{->>}[r]^{1+\nabla_2+1} \ar@<3pt>[d]^{\nabla_2 +\nabla_2}   & 3X \ar@<3pt>[d]^{\nabla_3} \\
     2X \ar@{->>}[r] \ar@<3pt>[u]^{\iota_2 + \iota_1} & X. \ar@<3pt>[u]^{\iota_2}
  }
$$
If $F$ is a cover of $\Eq(\nabla_2+\nabla_2))$, then $ \xymatrix { F \ar@<2pt>[r] \ar@<-2pt>[r] & 4X}$ is a weak kernel pair of $\nabla_2+\nabla_2$.
By assumption $ \xymatrix { F \ar@<2pt>[r] \ar@<-2pt>[r] & 3X}$ is a weak kernel pair of $\nabla_3$, so that $\lambda \mu$ is surjective. We then conclude that $\lambda$ is surjective and the existence of the quaternary operations $p$ and $q$ follows from Theorem $3$ in \cite{gr}.
\end{remark}

\section*{Acknowledgements}
The first author acknowledges partial financial assistance by Centro de Mate--m\'{a}tica da
Universidade de Coimbra---UID/MAT/00324/2013, funded by the
Portuguese Government through FCT/MCTES and co-funded by the
European Regional Development Fund through the Partnership
Agreement\\ PT2020.\\
 The second author acknowledges financial assistance by Fonds de la Recher--che Scientifique-FNRS Cr\'edit Bref S\'ejour \`a l'\'etranger 2018/V 3/5/033 - IB/JN - 11440, which supported his stay at the University of Algarve, where this paper was partially written.



\pagebreak
\noindent
Diana Rodelo \\
CMUC, Department of Mathematics\\
 University of Coimbra\\
3001--501 Coimbra, Portugal\\
Departamento de Matem\'atica, Faculdade de Ci\^{e}ncias e Tecnologia\\
 Universidade do Algarve, Campus de Gambelas\\
 8005--139 Faro, Portugal\\
drodelo@ualg.pt
\vspace{5mm}

\noindent
Idriss Tchoffo Nguefeu \\
Institut de Recherche en Math\'ematique et Physique\\
Universit\'e Catholique de Louvain\\
Chemin du Cyclotron 2, 1348 Louvain-la-Neuve, Belgium\\
idriss.tchoffo@uclouvain.be


\end{document}